\documentclass[12pt,a4paper]{amsart}
\usepackage{amsmath,amsthm,amssymb,amsfonts}

\usepackage{graphicx}
\usepackage{color}
\usepackage[active]{srcltx}
\usepackage[utf8]{inputenc}
\usepackage[english]{babel}
\usepackage{fourier}

\newtheorem{theorem}{Theorem}[section]
\newtheorem{lemma}[theorem]{Lemma}

\newtheorem{proposition}[theorem]{Proposition}
\newtheorem{corollary}[theorem]{Corollary}
\newtheorem{question}[theorem]{Question}
\theoremstyle{definition}
\newtheorem{definition}[theorem]{Definition}
\newtheorem{remark}[theorem]{Remark}
\numberwithin{equation}{section}
\newcommand{\N}{\mathbb{N}} %% Conjunto naturales:     \N
\newcommand{\Q}{\mathbb{Q}} %% Conjunto racionales:    \Q
\newcommand{\R}{\mathbb{R}} %% Conjunto reales:        \R
\newcommand{\C}{\mathbb{C}} %% Conjunto complejos:     \C
\newcommand{\D}{\mathbb{D}} %% Disco unidad:           \D
\newcommand{\e}{\varepsilon}

\def\Re{\hbox{\rm Re~}}

\title[Dirichlet approximation and universal      Dirichlet series]
{Dirichlet approximation and universal Dirichlet series}

\author[R. M. Aron]{Richard M. Aron}
\address{Department of Mathematical Sciences, Kent State University, Kent OH 44242, USA} \email{aron@math.kent.edu}
\author[F. Bayart]{Fr\'{e}d\'{e}ric Bayart}
\address{ Laboratoire de Math\'ematiques, Universit\'e Blaise Pascal, BP 10448, F-63000 Clermont-Ferrand, France} \email{Frederic.Bayart@math.univ-bpclermont.fr}
\author[Paul  Gauthier]{Paul M. Gauthier}
\address{D\'epartement de math\'ematiques et de statistique, Universit\'e de Montr\'eal, Montr\'eal QC, H3C3J7, Canada}
\email{gauthier@dms.umontreal.ca}
\author[Manuel Maestre]{Manuel Maestre}
\address{Departamento de An\'{a}lisis Matem\'{a}tico,
Universidad de Valencia, Doctor Moliner 50, 46100 Burjasot
(Valencia), Spain} \email{manuel.maestre@uv.es}
\author[Vassili Nestoridis]{Vassili Nestoridis}
\address{Department of Mathematics, University of Athens, 157 84 Panepistemiopolis, Athens, Greece} \email{vnestor@math.uoa.gr}

\thanks{
The first and fourth
authors were supported by MINECO MTM2014-57838-C2-2-P and  Prometeo II/2013/013; the third by NSERC and Entente France-Qu\'ebec. Partially supported also by the ``SQuaREs" program at the American Institute of Mathematics, Palo Alto and the ``Research in Pairs" program at the Mathematisches Forshungsinstitut, Oberwolfach.}


\date{}


\begin{document}

\begin{abstract}

We characterize the uniform limits of Dirichlet polynomials
on a right half plane. 
In the Dirichlet setting,   we find approximation results, with respect to the
Euclidean distance and {to} the chordal one as well, analogous to classical results of Runge,
Mergelyan and Vitushkin.
We also strengthen the
notion of universal      Dirichlet series.

 \end{abstract}

\maketitle

\section{Introduction}
\par

It is well known that the set of uniform limits of polynomials on the
closed unit disk $\overline \D$ is the disc algebra $A(\D),$ consisting of all continuous
functions $f: \overline \D\rightarrow \C,$ which are holomorphic in the open unit disk $\D.$ In \cite{MakNes},
for an arbitrary set $I,$ the set of uniform limits $A(\D^I)$ on $\D^I$  of
polynomials depending each one on a varying finite set of variables is
investigated. The space $A(\D^I)$ consists of all continuous functions $f:\overline \D^I\rightarrow \C$, where $\overline \D^I$ is endowed with the product topology, which are
separately holomorphic in $\D^I.$

The motivation of the present paper is to replace polynomials by Dirichlet
polynomials and to examine analogous approximation questions. First we
investigate the set of uniform limits of Dirichlet polynomials on the right
half plane $\C_+     =\{s\in \mathbb C: \Re s>0\}.$ This set coincides with the space of uniformly
continuous functions $f:\C_+      \rightarrow \C,$  representable by  convergent
Dirichlet series on $\C_+     .$ It is isomorphic to the space $A(\D^{\N_0})$ and to
the space of bounded uniformly continuous and complex Frechet
differentiable functions on the open unit ball of $c_0.$ This brings
together the results of [13] and [14] and it is the content of Section 2.

    Next we consider uniform approximation by Dirichlet
polynomials on compact subsets of $\C.$ For this, we need to
strengthen an approximation result from [3], valid for particular
compact sets said to be  "admissible". This is done in Section 3 and
automatically we strengthen the notion of universal      Dirichlet series (
[3],[5] ). It follows that the universal approximation by the partial sums of a
universal      Dirichlet series is valid on every compact set $K$ in $\C_{-}:=\{s
\in \mathbb C: \Re s \le 0 \}$  having connected complement, not only  on "
admissible " compact sets $K.$

   In Section 4, using the previous approximation result, we extend results of
Mergelyan, Runge and Vitushkin, replacing polynomials by
Dirichlet polynomials and rational functions by Dirichlet rational
functions. A Dirichlet polynomial has the form $P(s)=\sum_{j=1}^{n}a_jj^{-s}$   and a Dirichlet rational function has the form
$P_0(s)+P_1(1/(s-z_1))+...+P_n(1/(s-z_n)),$ where $P_j , j=0,1\ldots,n,$ are
Dirichlet polynomials . We also investigate the set of uniform limits of
Dirichlet polynomials on any straight line in $\C.$

In Section 5 we treat  analogous questions where the uniform
approximation is not meant with respect to the usual Euclidean distance on
$\C=\R^2$ but with respect to the chordal distance on $\C\cup \{\infty\}.$
This extends results of \cite{FraNePa} to the Dirichlet setting.

%%%%%%%%%%%%%%%%%%%%%%%%%%%%%%%%%%%%%%%%%%%%%%

\section{Background review; Closure of Dirichlet polynomials on a half plane }
\par

To any Dirichlet series $D=\sum_{n=1}^{\infty}\frac{a_n}{n^s}$, one can associate the following abscissas:
$$\sigma_c(D)=\inf\{ \Re s: \ \sum_{n=1}^{\infty}\frac{a_n}{n^s}\ \text{is convergent}\}.$$
$$\sigma_b(D)=\inf\{ \sigma : \ \sum_{n=1}^{\infty}\frac{a_n}{n^s}\ \text{is bounded on } \ \Re s\geq \sigma\}.$$
$$\sigma_u(D)=\inf\{ \sigma : \ \sum_{n=1}^{\infty}\frac{a_n}{n^s}\ \text{is uniformly convergent on } \ \Re s\geq \sigma\}.$$
$$\sigma_a(D)=\inf\{ \Re s: \ \sum_{n=1}^{\infty}\frac{a_n}{n^s}\ \text{is absolutely convergent}\}.$$

%\textcolor{red}{Note change of order in the above abscissas}

It is known that  $$\sigma_c(D)\leq \sigma_b(D)= \sigma_u(D)\leq \sigma_a(D)\leq \sigma_c(D)+1.$$

{The proof of inequalities relating $\sigma_c(D), \sigma_u(D)$ and $\sigma_a(D)$ can be found in \cite[Section 4.1, pag. 98]{QuQu13} or in \cite{Boas}. The equality $\sigma_b(D)= \sigma_u(D)$ was obtained by Bohr, and a proof can be found in \cite[Theorem 6.2.3, pag. 145]{QuQu13}.}
\par

Recall that $\mathcal{H}_\infty$ is the Banach space of all Dirichlet series $\sum_{n=1}^{\infty}\frac{a_n}{n^s}$ that converges to a bounded function $D(s)$ on   $\C_+$ endowed with the supremum norm
$$\|D\|_\infty=\sup_{s \in  \C_+}\big|  \sum_{n=1}^{\infty}\frac{a_n}{n^s}\big|.$$
\begin{definition} Let us denote by $\mathcal{A}( \C_+)$ the set of all Dirichlet series  $D(s)=\sum_{n=1}^{\infty}\frac{a_n}{n^s}$ which are convergent on $\C_+$
and define a uniformly continuous function on that half plane.
\end{definition}

\begin{proposition}\label{AisclosedinHinfty}
  $\mathcal{A}( \C_+)$ is a closed subspace of $\mathcal{H}_\infty$.
\end{proposition}

\begin{proof}
  Since the uniform limit of a sequence of uniformly continuous functions is uniformly continuous, it is enough to check that $\mathcal{A}( \C_+)$ is a  subset of $\mathcal{H}_\infty$.
  Let $D(s)=\sum_{n=1}^{\infty}\frac{a_n}{n^s} \in \mathcal{A}( \C_+)$.  By hypothesis there exists $\delta>0$ such that
  $|D(s_1)-D(s_2)|<1$ for every $s_1,s_2 \in \C_+$ such that $|s_1-s_2|<\delta.$
Let us take $k\in \mathbb{N}$ with $\frac{1}{k}<\delta$.
Since $\sigma_c(D)\leq 0$, we have   $\sigma_a(D)\leq 1$. Thus
  $$|D(s)|= \big|\sum_{n=1}^{\infty}\frac{a_n}{n^s}\big |\leq \sum_{n=1}^{\infty}\frac{|a_n|}{n^{1+\frac{1}{k}}}:=M<\infty,$$
  for every $s \in \C$ such that  $\Re s\geq 1+\frac{1}{k}$.
  On the other hand if $0<\Re s <1+\frac{1}{k}$, let $j_0$ be the smallest non-negative integer such that $\Re s+
\frac{j_0}{k}\geq 1+\frac{1}{k}$. Clearly $j_0\leq k+1$ and
$$
	|D(s)| \leq
	\sum_{j=0}^{j_0-1}\left|D\left(s+\frac{j}{k}\right)-D\left(s+\frac{j+1}{k}\right)\right|+
	\left|D\left(s+\frac{j_0}{k}\right)\right| \le j_0+M\leq k+1+M.
$$

\end{proof}

\begin{theorem}\label{CharactuniflimDirichletpoly} Given $f:\C_+\to\C$, the following are equivalent.
  \begin{enumerate}
    \item  $f$ is the uniform limit on $\C_+$ of a sequence of Dirichlet polynomials.
    \item  $f$ is represented by a Dirichlet series pointwise on $\C_+$ and $f$ is uniformly continuous on  $\C_+$.
  \end{enumerate}
\end{theorem}
\begin{proof}
(1) $\Rightarrow$ (2)\
  Let $f:\C_+\to\C$, be  the uniform limit on $\C_+$ of a sequence $(P_k)$ of Dirichlet polynomials $P_k(s)=\sum_{n=1}^{n_k}\frac{b_{n,k}}{n^s}$. Since $\mathcal{H}_\infty$ is a Banach space, $f$ belongs to  $\mathcal{H}_\infty.$ Thus, $f$ can be pointwise represented as $f(s)=\sum_{n=1}^{\infty}\frac{a_n}{n^s}$ and as the uniform limit of uniformly continuous functions,  $f$ too is uniformly continuous.

  (2) $\Rightarrow$ (1)  Conversely, let $f(s)=\sum_{n=1}^{\infty}\frac{a_n}{n^s}$ pointwise on $\C_+$ with $f$ uniformly continuous on  $\C_+$. Given $\varepsilon>0$, there exists $\delta>0$ such that if $s_1,s_2 \in \C_+$ satisfy $|s_1-s_2|<\delta$, then $|f(s_1)-f(s_2)|<\varepsilon/2$. {Since,} by Proposition \ref{AisclosedinHinfty},   $f$ is bounded on $\C_+$, we have  $\sigma_u(f)=\sigma_b(f)\leq 0$. Hence, given $\delta/2$, the polynomials $P_k(s)=\sum_{n=1}^{k}\frac{a_n}{n^s}$   converge uniformly to $f$ on $[\Re s \geq \delta/2]$. Thus if we set $Q_k(s)=P_k(s+\delta/2)=\sum_{n=1}^{k}\frac{a_n}{n^{\delta/2}}\frac{1}{n^s},$ then $\{Q_k(s)\}$ is a sequence of Dirichlet polynomials converging uniformly to $f(s+\delta/2)$ on $\C_+$ and we obtain
$$
\left|f(s)- \sum_{n=1}^{k}\frac{a_n}{n^{\delta/2}}\frac{1}{n^s}\right|\leq \left|f(s)-f(s+\frac{\delta}{2})\right|+\left|f(s+\frac{\delta}{2})- \sum_{n=1}^{k}\frac{a_n}{n^{\delta/2}}\frac{1}{n^s}\right| <
\frac{\varepsilon}{2}+\left\|f(\cdot+\frac{\delta}{2})-Q_k\right\|,
$$
for every $s\in \C_+$ and every $k$. Now choosing $k_0$ such that $\|f(\cdot+\frac{\delta}{2})-Q_k\|\leq \frac{\varepsilon}{2}$ for $k\geq k_0$, we arrive at
   $$\|f-Q_k\|\leq \varepsilon,$$
for every $k\geq k_0.$
  \end{proof}

Let $B_{c_0}$ denote the open unit ball in the space  $c_0$ of sequences which converge to zero and for a function $f:B_{c_0}\rightarrow\C,$  set  $\|f\|_{B_{c_0}}=\sup\{|f(x)|: x\in B_{c_0}\}.$ Endowed with this norm,  the space of all  functions on $B_{c_0}$ which are  uniformly continuous and complex Fr\'{e}chet differentiable is a Banach algebra, which we denote  by $\mathcal{A}_u( B_{c_0}).$

\begin{theorem}\label{AisAco}
  $\mathcal{A}( \C_+)$ is isometrically isomorphic to $\mathcal{A}_u( B_{c_0}).$
\end{theorem}

{Hedenmalm, Lindqvist and Seip} in \cite{HeLiSe} obtained that $\mathcal{H}_\infty$, the
Banach space of all Dirichlet series bounded and convergent on $\C_+$ is isometrically isometric to  ${H}_\infty(B_{c_0}),$
the Banach space of all bounded and  complex Fr\'{e}chet differentiable functions on the open unit ball of $c_0$.
{Even though their proof makes use of} other isometries, a careful analysis yields our claim, taking into consideration that
$B_{c_0}$ {is denoted by} $\D^\infty\cap c_0(\N)$ in \cite{HeLiSe}, {and using our formula}   \eqref{AuBc0},
below. 
On the other hand, in the very recent preprint \cite{DeFreMaSe}, another proof of that isometry can be found, but again the proof of our statement is only  implicit there and {one} needs to invoke  \eqref{AuBc0}. We present here {an explicit} proof of this fact since we shall need it in the proof  Theorem  \ref{AA}.

\begin{proof}
Let $\N_0^{(\N)}$ denote the set of  sequences of non-negative integers with finite support. {
 That is, $\N_0^{(\N)}$ consists of all $\alpha=(\alpha_1,\ldots,\alpha_N,0\ldots)$} where $N \in \N$ and $\alpha_j\in \N_0=\N\cup\{0\}$ for $j=1,\ldots,N$. By $z^\alpha$ we denote the monomial defined on $\C^\N$ by $z^\alpha=z_1^{\alpha_1}\ldots z_N^{\alpha_N}$. Obviously
 $z^\alpha$ belongs to  $\mathcal{A}_u( B_{c_0})$  since it is a complex Fr\'{e}chet differentiable function when restricted to $c_0$ that is uniformly continuous on $B_{c_0}$. But $z^\alpha$ also can be considered as a monomial in $N$-variables defined on $\C^N$.

A function $P:c_0\to \C$ is called an $m$-homogeneous polynomial if there exists  a continuous $m$-linear form $A:c_0^m\to \C,$ such that $P(x)=A(x,\ldots,x)$ for every $x \in c_0$. Clearly $P \in  \mathcal{A}_u( B_{c_0})$.  It is known (see e.g \cite{ACG}) that every function on  $\mathcal{A}_u( B_{c_0})$ is the uniform limit of linear combinations of continuous homogeneous polynomials. In other words
$$
	\mathcal{A}_u( B_{c_0}) =
		\overline{\text{span}\{P_m, \ \text{$m$-homogeneous polynomial}\, :\, m  \in \N_0\}}^{\|.\|_{B_{c_0}}}.
$$
But in \cite{Bogda} it was shown  that every homogeneous polynomial on $c_0$ is weakly uniformly continuous when restricted to the unit ball of $c_0$ and, by  \cite[Proposition 2.8, p. 90]{Dineen} and the fact that the canonical basis of $c_0$ is shrinking, every  such  polynomial is the uniform limit on $B_{c_0}$ of a sequence of polynomials of the form $P(z)=\sum_{\alpha \in J}c_\alpha z^\alpha$ where $J$ is a finite subset of $N_0^{(\N)}$ and $c_\alpha \in \C$ for every $\alpha \in J$. Thus we have
\begin{equation}\label{AuBc0}
	 \mathcal{A}_u(B_{c_0})=\overline{\text{span}\{z^\alpha\,:\, \alpha \in \N_0^{(\N)}\}}^{\|.\|_{B_{c_0}}}.
\end{equation}
The other ingredient needed for the proof is a  deep result of Harald Bohr \cite{Bohr} (see also \cite[pp. 115-117]{QuQu13}).  Let $p=(p_n)$ denote the increasing sequence of prime numbers. Given $\alpha=(\alpha_1,\ldots,\alpha_k,0\ldots)\in \N_0^{(\N)}$ we set $p^\alpha=p_1^{\alpha_1}\ldots,p_k^{\alpha_k}$ and $N \in \N,$ we denote by $\Lambda_N$ the family of multi-indexes $\alpha \in \N_0^{(\N)}$ such that $p^\alpha \leq N$. Then
$$
	\sup_{\Re s>0}\left|\sum_{n=1}^{N}\frac{a_n}{n^s}\right| =
		\sup_{z \in \D^k}\left|\sum_{\alpha \in \Lambda_N}a_{p^\alpha}z^\alpha\right|,
$$
for every $a_1, \ldots a_N \in \C$ and all $N,$ where the number $k$ of variables is the subscript of the biggest prime number $p_k$ less than or equal to $N$ (in other words, k is the number of primes less than or equal to $N$).
This results yields that the correspondence:
$$
	\frac{1}{n^s} \longmapsto z^\alpha, \quad n=p^\alpha,
$$
generates a mapping
$$
	\phi: \text{span}\left\{\frac{1}{n^s}:\, n \in \N\right\}\subset \mathcal{A}(\C_+)  \longrightarrow
		\text{span}\left\{z^\alpha\,:\, \alpha \in \N_0^{(\N)}\right\}\subset \mathcal{A}_u(B_{c_0}) ,
$$
which is an isometry and is both linear and multiplicative. Thus it can be extended to a {surjective isometry}
$$
	\tilde{\phi}: \overline{\text{span}\{\frac{1}{n^s}:\, n \in \N\}}^{\|.\|_\infty}  \longrightarrow \, \,\, \,\, \,\, \,							 \overline{\text{span}\{z^\alpha\,:\, \alpha \in \N_0^{(\N)}\}}^{\|.\|_{B_{c_0}}}.
$$

The conclusion follows from Theorem \ref{CharactuniflimDirichletpoly} which states that $\mathcal{A}(\C_+)=\overline{\text{span}\{\frac{1}{n^s}:\, n \in \N\}}^{\|.\|_\infty}$, and  from \eqref{AuBc0}.

\end{proof}

\begin{remark}
Since we can extend any uniformly continuous function on $\C_+$ to its closure $\overline{\C_+}$, any function in $\mathcal A(\C_+)$ extends uniformly to a uniformly continuous function on $\overline{\C_+}$. However,
there {exist} functions $f:\overline{\C_+}\to \C$ {that are}
uniformly continuous on $\overline{\C_+}$ and pointwise representable for a
certain $\sum_{n=1}^{\infty}\frac{a_n}{n^s}$ for every $s \in \C_+$, but {for which} nevertheless there
{is some} $t \in \R$ so that $\sum_{n=1}^{\infty}\frac{a_n}{n^{it}}$ does
not converge. {That is,}  in general $f(it)$ is not representable by this Dirichlet series.

Indeed, there exists $g(z)=\sum_{n=0}^{\infty}b_nz^n$ in the {disk algebra} such that $\sum_{n=0}^{\infty}b_n$ is not convergent. Hence, by considering $g$ in $ \mathcal{A}_u(B_{c_0})$ and taking  $f=\tilde{\phi}^{-1}(g)\in \mathcal{A}(C_+)$, we have $f(s)=\sum_{l=1}^{\infty}b_{l-1}\frac{1}{(2^{l-1})^s}$
for every $s \in \C_+$. But
$$\sum_{l=1}^{\infty}b_{l-1}\frac{1}{(2^{l-1})^0}=\sum_{n=0}^{\infty}b_n$$
does not converge and hence cannot represent the uniformly continuous  unique  extension of $f$ at 0.
\end{remark}

Given a non-empty set of indices $I$, we consider the compact Hausdorff space $\overline{\D}^I$, endowed with the product topology. Quite recently in \cite{MakNes}, Makridis and Nestoridis  introduced  the space $A(\overline{\D}^I)$ of all functions $f:\overline{\D}^I\to \C$ continuous on $\overline{\D}^I,$ endowed with the product topology, and separately holomorphic on $\D^I.$  The space $A(\overline{\D}^I)$  is a Banach subalgebra of the space $C(\overline{\D}^I)$ of all continuous functions on $\overline{\D}^I$ endowed with the supremum norm.
{We now show that whenever} $I$ is an infinite and countable set, this space is going to be isometrically isomorphic to the space of Dirichlet series $\mathcal{A}( \C_+)$.

\begin{theorem}\label{AA}
  For $I=\N,$ the  Banach algebra  $A(\overline{\D}^{\N})$ is isometrically isomorphic to $\mathcal{A}( \C_+)$ and also to $\mathcal{A}_u(B_{c_0})$.
\end{theorem}

\begin{proof}
The mapping
$$
	\psi: \text{span}\{z^\alpha\,:\, \alpha \in \N_0^{(\N)}\}\subset \mathcal{A}_u(B_{c_0}) \longrightarrow A(\overline{\D}^{\N}),
$$
generated by setting $\psi(z^\alpha)=z^\alpha,$ is an isometry which is into, linear and multiplicative. Hence $\psi$ can
be extended to another linear and multiplicative {surjective} isometry
$$
	\tilde{\psi}: \overline{\text{span}\{z^\alpha\,:\, \alpha \in \N_0^{(\N)}\}}^{\|.\|_{B_{c_0}}} \longrightarrow
		\overline{\text{span}\{z^\alpha\,:\, \alpha \in \N_0^{(\N)}\}}^{\|.\|_{C(\overline{\D}^I)}}.
$$
But we have already shown that
$$ \mathcal{A}_u(B_{c_0})=\overline{\text{span}\{z^\alpha\,:\, \alpha \in \N_0^{(\N)}\}}^{\|.\|_{B_{c_0}}},$$
Moreover, by \cite[Proposition 3.1]{MakNes}, given  $f \in A(\overline{\D}^{\N}),$ there exists a sequence $(P_n)$ of polynomials of the form
$$P(z)=\sum_{\alpha \in J}c_\alpha z^\alpha,$$
where $J$ is a finite subset of $\N_0^{(\N)},$ such that the sequence   $(P_n)$ converges uniformly on $\overline{\D}^{\N}$ to $f$.
Hence,
$$
	A(\overline{\D}^{\N})=\overline{\text{span}\{z^\alpha\,:\, \alpha \in \N_0^{(\N)}\}}^{\|.\|_{C(\overline{\D}^I)}}
.$$
We have obtained that $A(\overline{\D}^{\N})\overset{1}{\equiv} \mathcal{A}_u(B_{c_0})$. Finally, by Theorem \ref{AisAco}, $\mathcal{A}_u( B_{c_0})$ is isometrically isomorphic to
$\mathcal{A}( \C_+)$.
\end{proof}
\begin{corollary}
  Given a  function $f:\C\to \C$, the following are equivalent.
  \begin{enumerate}
    \item   On  every (right) half-plane,  $f$ is the uniform limit of Dirichlet polynomials.
    \item There exists a Dirichlet series $\sum_{n=1}^{\infty} a_n\frac{1}{n^s},$ which converges to $f$ at every point of\, $\C$.
  \end{enumerate}
\end{corollary}
\begin{proof}
  (2) ${\Rightarrow}$ (1).  Assume that $f(s)=\sum_{n=1}^{\infty} a_n\frac{1}{n^s}$ for all points $s$ of $\C$. Since $\sigma_a(f)\leq \sigma_c(f)+1$ we have that $\sum_{n=1}^{\infty} |a_n|\frac{1}{n^{{\rm Re}\,  s}}<\infty$ for every $s$. Hence the sequence of Dirichlet polynomials $\sum_{n=1}^{N} a_n\frac{1}{n^s}$ converges to $f$ uniformly on every half-plane.

     (1) ${\Rightarrow}$ (2). Fix $k$ a non-positive integer. With  a suitable change of variables, Theorem \ref{CharactuniflimDirichletpoly}, implies that there exists a Dirichlet series $\sum_{n=1}^{\infty} a_n(k)\frac{1}{n^s}$ such that
  $$f(s)=\sum_{n=1}^{\infty} a_n(k)\frac{1}{n^s},$$
  for every  $s$ in  $\C$ with $\Re s\geq k$. Thus all Dirichlet series {of $f$}
  coincide on $\C_+$. Now the uniqueness of coefficients of a convergent Dirichlet series on $\C_+$ implies that $a_n(k)=a_n(l)$ for every  non positive integers $k$ and $l$ and every $n\in \N$. Defining $a_n=a_n(k)$, we obtain
  $f(s)=\sum_{n=1}^{\infty} a_n\frac{1}{n^s},$ for every $s\in \C$.
\end{proof}

%%%%%%%%%%%%%%%%%%%%%%%%%%%%%%%%%%%%%%%%%%%%%%%%%%%%%%%%%%%%%%%%%%%%%%%%%%%%%%%%%%%%%%%%%%%%

\section{Universal Dirichlet series}

The {existence of the}
first universal Dirichlet series {was} established by the second author in \cite{Bayart}
under the assumption that the compact sets under consideration were ``admissible."
Further results on universal Dirichlet series  can be found in \cite{NePa}, \cite{DeMo}, \cite{Mo}.
Now, we improve most of those
results {by} relaxing the assumption that the compact set be admissible.

For $\sigma>0$ we denote $\|\sum_{n=1}^{\infty}a_n n^{-s}\|_\sigma=\sum_{n=1}^{\infty}|a_n|n^{-\sigma}$, 
and
$$
	D_a(\C_+) = \left\{\sum_{n=1}^{\infty}a_n n^{-s}\,:\, \left\|\sum_{n=1}^{\infty}a_n n^{-s}\right\|_\sigma<\infty, \ \text{for all}\ \sigma>0\right\}.
$$
We endow $D_a(\C_+)$ with the {Fr\'echet} topology induced by the semi-norms  $\|\cdot\|_\sigma, \sigma>0$ and we denote  $\C_-=\{s \in \mathbb C\, :\, \Re s<0\}.$

%\begin{lemma}
%  Let $K\subset \overline{\C_-}$ be compact with connected complement, $f\in D_a(\C_+),$ $g\in A(K)$, $\sigma>0$ and $\varepsilon>0$. Then there exists a Dirichlet polynomial $h$, $h(s)=\sum_{n=1}^{N}a_n n^{-s}$, $s \in \N$ such that
%  $$\|h-f\|_{C(K)}<\e \ \ \ \ \  \text{and}\ \ \ \|h-f\|_\sigma <\e.$$
%\end{lemma}

\begin{theorem}\label{thm:approxdirichlet}
Let $K\subset\overline{\C_-}$ be compact with connected complement, $f\in D_a(\C_+)$, $g\in A(K)$,
$\sigma>0$ and $\e>0$. Then there exists a Dirichlet polynomial $h=\sum_{n=1}^N a_nn^{-s}$ such that $\|h-g\|_{C(K)}<\e$ and $\|h-f\|_\sigma<\e.$
\end{theorem}

The proof of this theorem, inspired by \cite{BAGCHI} and by \cite[Section 11.5]{BM09}, needs some preparation. We first give two lemmas on entire functions of exponential type.

\begin{lemma}\label{lem:exponentialtype1}
Let {$f:\C \to \C$} be an entire function of exponential type and let  $\alpha$ be greater than the type. There exists a positive constant $C$ depending only on $f,$  such that if $u>1$ with $|f(u)|>0$ and $N$ is an integer satisfying $N\geq e^2\alpha(u+1),$ then there exists an interval $I\subset [u-1,u+1]$ of length $1/(2N^2)$ such that
$$
	x\in I\implies |f(x)|\geq \frac{|f(u)|}2+Ce^{-N}.
$$
\end{lemma}

A proof of this lemma can be found in \cite[Lemma 7.2.7]{QuQu13}.

\begin{lemma}\label{lem:exponentialtype2}
Let {$f:\C \to \C$} be a non-zero {entire} function of exponential type and let $\delta\in (0,1)$. Assume that
$$
	\limsup_{x\to+\infty}\frac{\log|f(x)|}x\geq 0.
$$
Then
$$
	\sum_{n\geq 1}\frac{|f(\log n)|}{n^{1-\delta}}=+\infty.
$$
\end{lemma}

\begin{proof}
Choose $\alpha$  greater than the type of $f$ and $0<\delta<1.$
Let $(x_j)$ be a sequence of  positive  real numbers going to infinity such that {for every $j\geq 1,$}\ $|f(x_j)|\geq e^{-\delta x_j/2},$  $x_j> 1$ and $\lfloor x_j\rfloor^2\ge e^2\alpha(x_j+1).$
Let $N_j=\lfloor x_j\rfloor^2$. By Lemma \ref{lem:exponentialtype1}, we can find an interval $I_j=\left[y_j,y_j+1/(2N_j^2)\right]$ contained in $[x_j-1,x_j+1]$ such that $x\in I_j\implies |f(x)|\geq e^{-\delta x_j/2}/2.$

By choosing $x_1$ sufficiently large, we may assume that
\begin{equation}\label{>1}
	e^{y_j+\frac{1}{2N_j^2}}-e^{y_j} > 1.
\end{equation}
Indeed,
$$
	e^{y_j+\frac{1}{2N_J^2}}-e^{y_j} \ge
		\min_{t\in[x_j-1,x_j+1]}\left(e^{t+\frac{1}{2N_j^2}}-e^t\right) =
		e^{x_j-1+\frac{1}{2N_j^2}}-e^{x_j-1} \ge e^{x_j-1+(x_j-1)^{-4}/2}-e^{x_j-1} =
				\frac{e^{(x_j-1)^{-4}/2}-1}{e^{1-x_j}},
$$
which tends to $+\infty$ by l'H\^opital's rule, establishing (\ref{>1}).
Now,
\begin{eqnarray*}
\sum_{\log n\in I_j}\frac{|f(\log n)|}{n^{1-\delta}}&\gg&\sum_{\log n\in I_j}\frac{e^{-\delta x_j/2}}{e^{(1-\delta)y_j}}\\
&\gg&e^{-(1-\delta/2)y_j} \textrm{card}\big\{n:\ \log n\in I_j\big\}\\
& =&e^{-(1-\delta/2)y_j}
	\textrm{card}\big\{n:\ n\in\big[e^{y_j},e^{y_j+\frac{1}{2N_j^2}}\big]\big\}\\
&\gg&e^{-(1-\delta/2)y_j}\left(e^{y_j+\frac 1{2N_j^2}}-e^{y_j}\right)
	\quad \mbox{by} \quad (\ref{>1})\\
&\gg&\frac{e^{\delta y_j/2}}{y_j^4}.
\end{eqnarray*}
But this last quantity goes to infinity as $j$ goes to infinity.
\end{proof}

We also need a result from {the} geometry of function spaces. A proof of it can be found in \cite[Lemma 11.11]{BM09}.

\begin{lemma}\label{lem:geometry}
Let $X$ be a locally convex topological vector space and let $(x_n)_{n\geq 1}$ be a sequence in $X$. Assume that $\sum_{n=1}^{+\infty}|\langle x^*,x_n\rangle|=+\infty$ for every  nonzero continuous linear functional $x^*\in X^*$. Then, for  every  $N\in\N$, the set $\left\{\sum_{n=N}^M a_n x_n;\ M\geq N, |a_n|\leq 1\right\}$ is dense in $X$.
\end{lemma}
\begin{proof}[Proof of Theorem \ref{thm:approxdirichlet}]
Let $K\subset\overline{\C_-}$ be compact with connected complement, let $P\in D_a(\C_+)$ be a Dirichlet polynomial, $g\in A(K)$, $\sigma>0$ and $\e>0$. We set $\delta=\sigma/2$ and we choose $N$ bigger than the degree of $P$ and such that
$$\sum_{n\geq N}\frac 1{n^{1+\delta}}<\e.$$
Let $\phi$ be a nonzero continuous linear functional on $A(K)$. By the Hahn-Banach and the Riesz representation theorems, there exists a nonzero (complex) measure $\mu$ with support contained in $K$ such that, for any $u\in A(K)$,
$$\langle \phi,u\rangle=\int_K ud\mu.$$
We intend to show that $\sum_{n\geq 1}|\langle \phi,n^{-s}\rangle|=+\infty$. Observe that
$$\phi(n^{-s})=\int_K e^{-s\log n}d\mu=\mathcal L_\mu(\log n),$$
where $\mathcal L_\mu$ is the Laplace transform of $\mu$. Observe that $\mathcal L_\mu$ is nonzero. Indeed,
differentiating under the integral sign and evaluating at zero,
$$\mathcal L_\mu^{(k)}(0)=(-1)^k \langle \phi, z^k\rangle.$$
Since $K$ has connected complement, polynomials are dense in $A(K)$ and since $\phi\neq 0$, it follows that $\mathcal L_\mu^{(k)}(0)\neq 0$ for some $k$. On the other hand, it is well-known (see for instance \cite[Lemma 11.15]{BM09}) that $\mathcal L_\mu$ is an entire function of exponential type and that
$$\limsup_{x\to+\infty}\frac{\log|\mathcal L_\mu(x)|}x\geq 0.$$
Applying Lemma \ref{lem:exponentialtype2}, we get
$$\sum_{n\geq 1}|\langle \phi,n^{-s-1+\delta}\rangle|=\sum_{n\geq 1}\frac{|\mathcal L_\mu(\log n)|}{n^{1-\delta}}=+\infty.$$
By Lemma \ref{lem:geometry}, this implies that $\left\{\sum_{n=N}^M a_n n^{-(1-\delta)}n^{-s};\ M\geq N,\ |a_n|\leq 1\right\}$ is dense in $A(K)$. In particular, there exists $h_0=\sum_{n=N}^M a_n n^{-(1-\delta)}n^{-s}$, $|a_n|\leq 1$ such that
$$\|h_0-(g-P)\|_{C(K)}<\e.$$
Setting $h=h_0+P$, we thus have $\|h-g\|_{C(K)}<\e$ and
$$\|h-P\|_\sigma=\|h_0\|_\sigma\leq \sum_{n=N}^{+\infty}\frac{1}{n^{1-\delta+\sigma}}<\e.$$
We have shown the theorem for the case that $f$ is a Dirichlet polynomial. The general case follows trivially.
\end{proof}

\begin{remark}
In the above theorem, translating $K$ and reducing $\sigma$ if necessary, we may in fact choose $K\subset\{s: \Re s<\sigma\}.$
Using Cauchy's formula, we deduce that, for every compact $L\subset\overline{\mathbb C_-},$ for every $f\in D_a(\mathbb C_+)$, $g$ entire, $\sigma,\e>0$, and $N\in\mathbb N$, there exists a Dirichlet polynomial $h$ such that
$$\sup_{0\leq l\leq N}\|h^{(l)}-g^{(l)}\|_{C(L)}<\e\textrm{ and }\|h-f\|_\sigma<\e.$$
\end{remark}

A consequence of Theorem \ref{thm:approxdirichlet} is the following result.

\begin{theorem}\label{UniversalDirichletseriesth}

There exists a Dirichlet series D,  absolutely convergent in $\C_+,$ with partial sums $S_N D,$ such that, for every entire function g there exists a sequence $( \lambda_n),$ so
that for every $\ell$ in \{0,1,2,\ldots\} the derivatives
$S_{\lambda_n}D^{(\ell)}$ converge to $g^{(\ell)}$
uniformly on each compact subset of the
closed left half plane, as n goes to
infinity.   The set of such Dirichlet series is dense and a $G_\delta$ in the space of Dirichlet series absolutely convergent in $\C_+,$  and it  contains a
{dense vector} subspace of it, apart from 0.
\end{theorem}

\begin{proof}
Let $K_m=[-m,0]\times[-m,m]$ and let $(g_k)$ be an enumeration of all polynomials with coefficients in $\Q+i\Q$. We endow the space $H(\C)$
 of entire functions with the following distance:
$$d_m(f,g)=\sum_{l=0}^{+\infty}2^{-l}\min\left(1,\|f^{(l)}-g^{(l)}\|_{C(K_m)}\right).$$
For $n,m,m,s\geq 1$, we define
$$E(n,k,m,s)=\{h\in D_a(\C_+): d_m(S_n h,g_k)<1/s\}.$$
By Remark 3.5, the set $\bigcup_n E(n,k,m,s)$ is dense and it is also clearly open in $D_a(\C_+)$.
We claim that each $D$ in $ \bigcap_{k,m,s}\bigcup_n E(n,k,m,s)$ satisfies the conclusion of the theorem. Indeed, let $g$ be any entire function.
For each $m\geq 1$, there exists $k_m$ such that $d_m(g,g_{k_m})<4^{-m}/2$. Then,
since $D\in \bigcup_n E(n,k_m,m,2\cdot 4^m)$, we know from the triangle inequality that there exists an integer $\lambda_m$
such that $d_m(S_{\lambda_m}D,g)<4^{-m}$. By the definition of $d_m$, this implies that, for any $l\leq m$,
$$\|S_{\lambda_m}D^{(l)}-g^{(l)}\|_{C(K_m)}<2^{-m}.$$
Since any compact set of $\overline{\mathbb C_-}$ is contained in all $K_m$ for any $m$ large enough, we are done. The construction of a dense vector subspace follows from the methods of \cite{BaGroNePa}.
 \end{proof}

This result improves \cite[Theorem 10]{BaGroNePa} where the set $K$ is assumed to be admissible, while in Theorem \ref{UniversalDirichletseriesth} of the present $K$ is any compact subset of $\overline{\C_-}$ with connected complement.

\begin{remark}
We should mention that S. Gardiner and M. Manolaki have recently been applying potential theory
techniques to study properties of universal Dirichlet series in \cite{GM}. This promising approach has yielded results that  apparently
cannot be answered using standard complex analysis.
\end{remark}

.

%%%%%%%%%%%%%%%%%%%%%%%%%%%%%%%%%%%%%%%%%%%%%%%%%

\section{Mergelyan and Runge type theorems in the Dirichlet setting }

Another consequence of Theorem \ref{thm:approxdirichlet} is the following result.
\begin{lemma}\label{Dirichlet:Mergelyan}
  Let $K \subset \C$ be compact with connected complement, $f\in A(K)$ and $\e>0$. Then there exists  a Dirichlet polynomial $P(s)=\sum_{n=1}^{N}a_n \frac{1}{n^s}$ such that
  $\|f-P\|_{C(K)}<\e$.
\end{lemma}

It follows easily from Lemma \ref{Dirichlet:Mergelyan} and Runge's theorem on polynomial approximation that if  $K \subset \C$ is an arbitrary compact set, then the set of uniform limits of Dirichlet polynomials on $K$ is the same as the set $P(K)$ of uniform limits of (algebraic) polynomias on $K.$ The family $P(K)$ is known to consist precisely of   
the set of functions on $K$ having an extension in $A(\widehat{K})$ where $\widehat{K}$ is the complement of the unbounded component of $K^c=\C\setminus K$. Indeed, if a function $f$ defined on $K$ is the uniform limit of a sequence $\{p_n\}$ of polynomials, then the sequence $\{p_n\}$ is uniformly Cauchy on $K$ and so, by the maximum principle it is also uniformly Cauchy on $\widehat K.$ Thus, the sequence  $\{p_n\}$ converges uniformly on $\widehat K$ to some function $\widehat f,$ which is in $A(\widehat K)$ and whose restriction to $K$ is $f.$ Conversely, if $f$ extends to a function $\widehat f\in A(\widehat K),$ then, by Mergelyan's theorem, $\widehat f$ is the uniform limit on $\widehat K$ of a sequence of polynomials and {\em a fortioru} $f$ is the uniform limit on $K$ of the same sequence of polynomials.

Another easy consequence of     Lemma \ref{Dirichlet:Mergelyan} is the following result.

\begin{theorem}
  Let $\Omega \subset \C$ be an open simply connected set and $H(\Omega)$ be the space of holomorphic functions on $\Omega$ endowed with the topology of uniform convergence on compacta. Then, Dirichlet polynomials are dense in $H(\Omega)$.
\end{theorem}

Mergelyan's theorem, initially valid for compact sets with connected complement, can be extended to compact sets $K$ such that  $K^c$ has a finite numbers of components, where the approximation will not be achieved by polynomials but by rational functions ( See  \cite[Exercise 1, Chapter 20,  p.394]{Rudin}). In analogy to this result we prove the following one.

\begin{theorem}\label{LaurentDirichlet}
  Let $\Omega \subset \C$ be a domain bounded by a finite number of disjoint Jordan curves. Let $\{z_1,\ldots, z_N\}$ be a set containing exactly one point from every bounded component of $\overline{\Omega}^c$. Let $f \in A(\Omega)$ and $\e>0$. Then there exist Dirichlet polynomials $P_0,P_1,\ldots, P_N$ so that the ``rational Dirichlet" function
  $$R(s)=P_0(s)+P_1\left(\frac{1}{s-z_1}\right)+\ldots+P_N\left(\frac{1}{s-z_N}\right)$$
  satisfies
  $$ \|f-R\|_{C(\overline{\Omega})}<\e.$$
\end{theorem}

\begin{proof}
 Let  $V_j$, $j=0,\ldots,N$ be the components of $\big(\C\cup\{\infty\}\big)\setminus \overline{\Omega}$ and $\infty \in V_0$.
By the well known ``Laurent Decomposition" (see for example \cite{CoNePa}),
$$f=f_0+f_1+\ldots+ f_N$$
where $f_j \in H(V_j^c)$ for $j=0,\ldots,N$ and $f_1(\infty)=\ldots= f_N(\infty)=0$.

The complement of $V_0^c$ is $V_0$ which is connected. So by Lemma \ref{Dirichlet:Mergelyan} there exists a Dirichlet polynomial $P_0$ so that
$$\|P_0-f\|_{C(\overline{\Omega})}\leq \|P_0-f\|_{C(V_0^c)}<\frac{\e}{N+1}.$$
If we do the inversion $z\to \frac{1}{z-z_j}$ ($j=1,\ldots,N$) then the image of $V_j^c$ is a compact set with connected component. Thus, by   Lemma \ref{Dirichlet:Mergelyan} the function $f_j(\frac{1}{w}+z_j)$ can be approximated by a Dirichlet polynomial $P_j(w)$ so that
$$
\|f_j(z)-P_j(\frac{1}{z-z_j})\|_{C(\overline{\Omega})}\leq \|f_j(z)-P_j(\frac{1}{z-z_j})\|_{C(V_j^c)}<\frac{\e}{N+1}.$$
The triangle inequality yields the result.
\end{proof}

For an arbitrary compact set $K\subset\C,$ we denote by $A(K)$ the class of functions $f:K\rightarrow\C,$ that are continuous on $K$ and holomorphic on the interior of $K.$ 
For a compact set $K\subset\C,$ clearly the set $P(K)$ of functions $f:K\rightarrow\C,$ which can be uniformly approximated by polynomials, is contained in the class $A(K).$ 
Mergelyan's theorem gives a characterization of those compact sets $K\subset\C$ such that $P(K)=A(K).$  Let us call such sets Mergelyan sets. Mergelyan's theorem then states that $K$ is  a Mergelyan set if and only if $K^c$ is connected. The analogous problem for rational approximation was solved by Vitushkin  (see, for example \cite{BG}). For compact $K\subset\C,$ denote by $R(K)$ the set of functions $f:K\rightarrow\C$ which are uniform limits of rational functions. Again, clearly $R(K)\subset A(K).$ 
Vitushkin characterized, in terms of continuous analytic capacity, those compact $K\subset\C$ for which $R(K)=A(K).$ Let us call such sets Vitushkin sets. Since rational functions are well-defined on the Riemann sphere, we could also consider uniform approxmation of functions $f:K\rightarrow\C,$ defined on compact subsets $K\subset\C\cup\{\infty\},$ by rational functions whose poles lie outside of $K.$  However, for simplicity, we shall confine our study to compact subsets of $\C.$

An example of a Vitushkin set is the closure $\overline\Omega$ of a domain $\Omega\subset\C$ bounded by finitely many disjoint Jordan curves. We can use this example to formulate an analogue of the previous theorem for arbitrary Vitushkin sets.
If $P$ is a Dirichlet polynomial and $a\in\C$ we define the rational Dirichlet function $P_a(s)=P(1/(s-a))$ and if $a=\infty,$ we set $P_\infty=P.$

\begin{theorem}
Let $K \subset\C$ be a Vitushkin compactum. Let $Z$ be a set which meets every component of $\overline\C\setminus K.$ Let $f \in A(K)$ and $\e>0$. Then, there exist finitely many points $z_1,\ldots,z_n\in Z$ and Dirichlet polynomials $P_{z_1},\ldots, P_{z_N}$ so that the ``rational Dirichlet" function
$$
	R(s)=P_{z_1}(s)+\ldots+P_{z_N}(s)
$$
 satisfies
  $$ \|f-R\|_{C(K)}<\e.$$
\end{theorem}

\begin{proof} By Vitushkin's theorem, there is a rational function $g$ such that $ \|f-g\|_{C(K)}<\e/2.$ Since rational functions composed with M\"obius transformations are again rational functions, we may assume that $K\subset\C.$ Let $\Omega\subset\C$ be a domain bounded by finitely many disjoint Jordan curves, such that $\Omega$ contains $K$ and all poles of $h$ lie outside of $\overline\Omega.$ By the previous theorem, there is a rational Dirichlet function $R$ such that $ \|g-R\|_{C(\overline\Omega)}<\e/2.$ The triangle inequality yields the desired approximation.
\end{proof}

The usual notation for the set of uniform limits of rational funcitons having no poles on $K$ is $R(K)$ and we have defined $K$ to be a Vitushkin set if $R(K)=A(K).$ Let us denote by $RD(K)$ the uniform limits of Dirichlet rational functions having no singularities on $K.$
We have shown that $RD(K)=A(K)$ if $R(K)=A(K),$ that is, if the compact set $K$ is a Vitushkin
compactum. In fact, this condition is also necessary, since for the converse we may approximate Dirichlet polynomials by
regular polynomials using Runge's  theorem.

Examples. Theorem \ref{LaurentDirichlet} asserts that the closure  $\overline\Omega$ of a domain  $\Omega$  bounded by finitely many disjoint Jordan curves is  a Vitushkin set. There are many other interesting examples of Vitushkin sets. For example, if the diameters of the complementary components of $K$ (even though there may be infinitely many such components) are bounded away from zero or if the compact set has area measure zero, then it is a Vitushkin set.

If we consider any arbitrary domain  $\Omega \subset \C$, then there is an exhausting family of compact subsets
$K_m$, $m=1,2,\ldots$ of  $\Omega$, such that every $K_m$ is bounded by finite number of disjoint Jordan curves. Thus, $K_m$ can be chosen as  $\overline{\Omega}$  in Theorem \ref{LaurentDirichlet} and Theorem \ref{LaurentDirichlet} can be applied. Let $A$ be a set containing a point in each component of
$\big(\C\cup\{\infty\}\big)\setminus \overline{\Omega}$. (In fact let $A_m=\{z_{1,m},\ldots, z_{N_m,m}\}$ be a set containing exactly one point from each bounded component of the set $K_m^c$. Then it suffices to take $A=\cup_{m=1}^\infty A_m$). Then one easily derives from Theorem \ref{LaurentDirichlet} the following Runge type theorem.

\begin{theorem}
  Let  $\Omega \subset \C$ be a domain and let $A=\cup_{m=1}^\infty A_m$ be as above. We consider $H(\Omega)$ endowed with the topology of uniform convergence  on compacta. Then ``rational Dirichlet" functions of the form 
$$
	R(s)=P_0(s)+\sum_{j=1}^{N}P_j\left(\frac{1}{s-z_j}\right),$$
$N \in \N$, $z_j \in A$ and $P_j(s)=\sum_{n=1}^{N_j}a_{n,j}\frac{1}{n^s}$ are dense in $H(\Omega)$.
\end{theorem}

In addition to considering uniform approximation by Dirichlet polynomials on compacta, one can raise the question of uniform approximation on (unbounded) closed sets. We restrict our comments to the case where the closed set is a line.
Let $E$ be a straight line in $\C.$ Denote by $X(E)$ the family of uniform limits of Dirichlet polynomials on $E.$

1) If $E=\R,$ it is easy to see that $X(\R)$ coincides wth the set of Dirichlet polynomials.  Indeed, the only Dirichlet polynomials which are bounded on the real line are the constant polynomials. Thus, if a sequence $\{P_j\}$ of Dirichlet polynomials is uniformly Cauchy on the real line, then for some natural number $n_0,$ we have that $P_{n_0}-P_j$ is a constant $c_j,$ for each $j\ge n_0,$ so $P_{n_0}-P_j$ converges to a constant $c.$ Thus, if $P_j$ converges uniformly on the real line to a function $f,$ then $f$ is the Dirichlet polynomial 
$P_{n_0}-c.$ The same for $E$ any line which is not vertical.

2) If $E=i\R,$ then by Theorem \ref{CharactuniflimDirichletpoly}, one can see that functions in  $X(i\R)$ are the boundary values of functions in  $\mathcal{A}( \C_+).$ That is,  $\varphi\in X(i\R)$ if and only if there exists a function $f\in \mathcal{A}( \C_+),$ such that
$$
	\varphi(i\tau) = \lim_{\sigma\rightarrow 0}f(\sigma+i\tau),  \quad -\infty<\tau<+\infty.
$$

%%%%%%%%%%%%%%%%%%%%%%%%%%%%%%%%%%%%%%%%%%%%%%%%%%%%

\section{Spherical Approximation}

In this section we try to extend the previous results when the uniform approximation is not meant with respect to the usual Euclidean distance but it is meant with respect to the chordal distance $\chi$.

We recall that the chordal distance $\chi$ is defined as follows.
$$
	\chi(a,b) =
\left\{
		\begin{array}{lll}
			\frac{|a-b|}{\sqrt{1+|a|^2}\sqrt{1+|b|^2}}		&	\mbox{for}	&	a,b \in \C,\\
			\frac{1}{\sqrt{1+|a|^2}}	&	\mbox{for}	&	a \in \C, b=\infty.
		\end{array}
\right.
$$
Obviously $\chi(a,b)\leq |a-b|$ for $a,b \in \C$.

Let us start with a compact set $K\subset \C$ with connected complement. By Mergelyan's Theorem every Dirichlet polynomial can be approximated by polynomials uniformly in $|.|$. Since  $\chi(a,b)\leq |a-b|$  for all $a,b  \in \C$, it follows that every Dirichlet polynomial can be approximated  $\chi$-uniformly on $K$ by polynomials. Lemma \ref{Dirichlet:Mergelyan} implies that every polynomial can be approximated on $K$ by Dirichlet polynomials, Since $\chi(a,b)\leq |a-b|$, it follows that every polynomial can be approximated $\chi$-uniformly on $K$ by Dirichlet polynomials. Thus we have the following proposition.

\begin{proposition}\label{chordal=uniform}
  Let $K\subset \C$ be a compact set with connected complement. Then the set of $\chi$-uniform limits on $K$ of Dirichlet polynomials coincides with the set of $\chi$-uniform limits on $K$ of polynomials.
\end{proposition}

If $K,$ is a closed Jordan domain,  the set of $\chi$-uniform limits on $K$ of polynomials is precisely the family of mappings
$\widetilde{A}(K)=\{f:K\to \C\cup\{\infty\}\ \text{continuous,  } \ f\equiv \infty \ \text{or}\ f(K^{\circ})\subset \C \ \text{and} \ f_{|K^{\circ}}\ \text{ holomorphic}\}.$ See  \cite{FraNePa} and the references therein.
Combining this  with the above proposition, we obtain the following result, which was proved in \cite{NeIPa}.

\begin{theorem}\label{xhilimitDirichletpolynomials}
  Let $K$ be the closure of a Jordan domain. Then the set of $\chi$-uniform limits on $K$ of Dirichlet polynomials is precisely $\widetilde{A}(K).$
\end{theorem}

For an arbitrary compact subset $K\subset\C,$ let $\widetilde{A}(K)$ be the family of continuous mappings $f:K\to \C\cup\{\infty\},$ such that  for every component $V$ of $K^{\circ}$ either $f_{|V}\equiv \infty$ or $f(V)\subset \C$ and $f_{|V}$ is holomorphic. In case $K$ is a closed Jordan domain, this is consistent with our previous definition of  $\widetilde{A}(K).$ It is well known that if $f:K\rightarrow\C\cup\{\infty\}$ is a
$\chi$-unifom limit of functions holomorphic on $K,$ then $f\in \widetilde{A}(K).$ In particular, this is the case if $f$ is a $\chi$-uniform limit on $K$ of Dirichlet polynomials on $K$ (see for example \cite{FraNePa}). The converse has been proved in the particular case where $K$ is the closure of a Jordan domain according to Theorem \ref{xhilimitDirichletpolynomials}, but it is open in the general case of compact sets sets $K$ with connected complement (see \cite{FraNePa}). Thus we have the

\begin{question}
 Let $K\subset \C$ be a compact set with $K^c$ connected. Let $f\in\widetilde{A}(K)$ and  $\e>0$; is it then true that there exists a Dirichlet polynomial  $P(s)=\sum_{n=1}^{N}a_n \frac{1}{n^s}$ such that $\chi(f(s),P(s))<\e$ for all $s \in K$?
 \end{question}

By Proposition \ref{chordal=uniform}, this problem is equivalent to that of finding, for every $\e>0,$ an algebraic polynomial $P$ such that   $\chi(f(s),P(s))<\e$ for all $s \in K.$ This, in turn, is equivalent to finding, for every $\e>0,$  a rational function $R,$ pole-free on $K,$ such that  $\chi(f(s),R(s))<\e$ for all $s \in K.$ Indeed, one direction is obvious, since every polynomial is a rational function pole-free on $K.$ Suppose, conversely, that there is a rational function $R$ pole-free on $K$ such that $\chi(f(s),R(s))<\e/2,$ for every $s\in K.$ By Runge's theorem, there is a polynomial $P,$ such that $|R(s)-P(s)|<\e/2,$ for every $s\in K.$ Hence, for every $s\in K,$ we have
$$
	\chi(f(s),P(s))\le\chi(f(s),R(s))+\chi(R(s),P(s))<\chi(f(s),R(s))+|R(s)-P(s)|<\e,
$$
which establishes the claim.

Combining Theorem \ref{xhilimitDirichletpolynomials} with a Laurent Decomposition (\cite{CoNePa}), we obtain the following Theorem \ref{ChiLimitLaurentDirichlet}. We omit the  proof, which is identical with the proof of Theorem 3.3 in \cite{FraNePa}, with the only difference that instead of polynomials we use Dirichlet polynomials.

\begin{theorem}\label{ChiLimitLaurentDirichlet}
  Let $\Omega\subset \C$ be a domain bounded by a finite set of disjoint Jordan curves. Let $K=\overline{\Omega}$ and $A$ be a set containing a point from every bounded component of $\C\setminus K$. Let  $f:K\to \C\cup\{\infty\}$ be a continuous mapping. Then the following are equivalent.
  \begin{enumerate}
    \item The function $f$ is the $\chi$-uniform limit on $K$ of a sequence of functions of the form $P_0(s)+\sum_{j=1}^{N}P_j(\frac{1}{s-z_j})$ where $P_j$ are Dirichlet polynomials and $z_j \in A$, $j=1,\ldots, N$, $N\in \N$ and $P_0$ is a Dirichlet polynomial.
    \item The function $f$ is the $\chi$-uniform limit on $K$ of a sequence of functions of the form $P_0(s)+\sum_{j=1}^{N}P_j(\frac{1}{s-z_j})$ where  $z_j \in \C\setminus K$, $j=1,\ldots, N$, $N\in \N$ and all $P_j$   are Dirichlet polynomials.
    \item $f\in\widetilde{A}(K).$
  \end{enumerate}
\end{theorem}

Since every domain has an exhaustion by compact sets $K_m$, $m=1,2,\ldots$ of the form of the set $K$ in Theorem
 \ref{ChiLimitLaurentDirichlet} we obtain the following Runge type theorem.

 \begin{theorem}
    Let $\Omega$ be a   domain in $\C$ and $A$ be a set containing $\infty$ and a point from each bounded component of\, $\C\setminus \Omega$. Let $f\equiv \infty$ or $f:\Omega \to \C$ be  holomorphic. Then $f$ is the $\chi$-uniform limit on each compact subset of\, $\Omega$ of functions of the form $h=P_0(s)+\sum_{j=1}^{N}P_j(\frac{1}{s-z_j})$, where $P_j$ are Dirichlet polynomials for $j=0,\ldots, N$,  $z_j\in A\setminus\{\infty\}$, for $j=1,\ldots, N$ and $N \in \N$.

   Conversely every $\chi$-uniform limit on each compact subset  of $\Omega$ of functions of the form of $h$ is either constant equal to $\infty$ or holomorphic in $\Omega$.
         \end{theorem}

\begin{remark}

Earlier, we made some remarks concerning uniform approximation on lines by Dirichlet polynomials. Now we consider $\chi$-uniform approximation on the real line.
Denote by $X_\chi(\R)$ the $\chi$-uniform limits on $\R$ of Dirichlet polynomials. If $f\in X_\chi(\R),$ what can we say about $f$ apart from the fact that it is $\chi$-continuous? The constant function $f\equiv \infty$ can be thus approximated by the Dirichlet polynomials $D_n(s)=n.$ Trivially, every Dirichlet polynomial can be so approximated. Are there other examples?

Consider the special case of Dirichlet series $D(s),$ which are $\chi$-uniformly convergent on $\R.$   We claim that $\Sigma(s)=\sum_{n=1}^\infty n^{-s}$ is such a series. It $\chi$-converges pointwise to the function $\Sigma(\sigma)=\zeta(\sigma),$ for $\sigma>1$ and $\Sigma(\sigma)=\infty,$ for $\sigma\le 1.$ The sum $\Sigma(\sigma)$ is $\chi$-continuous on $\R.$ We wish to show that the series converges $\chi$-uniformly. Fix $\epsilon>0$ and choose $\sigma_\epsilon>1$ such that $\chi(\Sigma(\sigma),\infty)<\epsilon,$ for $\sigma\le\sigma_\epsilon.$ Choose $N_1$ such that $\chi(\sum_{n=1}^Nn^{-\sigma_\epsilon},\infty)<\epsilon,$ for all $N\ge N_1.$
It follows that  $\chi(\sum_{n=1}^Nn^{-\sigma},\infty)<\epsilon,$ for all $\sigma\le\sigma_\epsilon$ and all $N\ge N_1.$  Now, choose $N_2$ such that $|\sum_{n=1}^Nn^{-\sigma}-\zeta(\sigma)|<\epsilon,$ for all $N\ge N_2$ and all $\sigma\ge\sigma_\epsilon.$ Now set $N_\epsilon=\max\{N_1,N_2\}.$ Then, $N\ge N_\epsilon$ implies that $\chi(\sum_{n=1}^\infty n^{-\sigma},\Sigma(\sigma))\le 2\epsilon,$ for all $\sigma\in\R.$ We have shown that the Dirichlet series $\Sigma(\sigma)$ converges $\chi$-uniformly on $\R.$ The same holds for every Dirichlet series $D(s)$ with non-negative coefficients. It converges $\chi$-uniformly on all of $\R.$ It is finite for $\sigma>\sigma_c(D)$ and it is $\infty$ for $\sigma\le\sigma_c(D).$ This holds even if $\sigma_c(D)=\pm\infty.$

\end{remark}

\end{document}